%% file: MainFile.tex
\documentclass[10pt]{amsart}

\usepackage{anysize}
\usepackage[utf8x]{inputenc}
\usepackage[main=english,czech]{babel}
\usepackage{amssymb, amsmath, amsthm}
\usepackage{enumitem}
\usepackage{mathtools}
\usepackage{url}
\usepackage[all,cmtip]{xy}
\usepackage[hidelinks,bookmarksnumbered]{hyperref}
\usepackage{cleveref}

\let\cal\mathcal
\def\AA{{\cal A}}

\def\CC{{\cal C}}
\def\DD{{\cal D}}
\def\EE{{\cal E}}
\def\FF{{\cal F}}

\def\II{{\cal I}}

\def\PP{{\cal P}}

\def\SS{{\cal S}}
\def\TT{{\cal T}}

\let\blb\mathbb

\def\bK{{\blb K}}

\def\bZ{{\blb Z}}

\def\bN{{\blb N}}

\def\bZ{{\blb Z}}

\let\frak\mathfrak

\def\fb{\frak{b}}

\title{Preresolving subcategories and derived equivalences}
\author{Ruben Henrard}
\address{Ruben Henrard \\ Universiteit Hasselt \\ Campus Diepenbeek \\ Departement WNI \\ 3590 Diepenbeek \\ Belgium}
\email{Ruben.Henrard@uhasselt.be}
\author{Adam-Christiaan van Roosmalen}
\address{Adam-Christiaan van Roosmalen \\ Universiteit Hasselt \\ Campus Diepenbeek \\ Departement WNI \\ 3590 Diepenbeek \\ Belgium}
\email{Adam-Christiaan.vanRoosmalen@UHasselt.be}
\subjclass[2020]{Primary: 18E05, 18G10, 18G80}
\keywords{(Pre)resolving subcategories, (one-sided) exact categories, derived categories}

\newtheorem{theorem}{Theorem}[section]
\newtheorem{proposition}[theorem]{Proposition}
\newtheorem{lemma}[theorem]{Lemma}
\newtheorem{corollary}[theorem]{Corollary}

\theoremstyle{definition}
\newtheorem{definition}[theorem]{Definition}
\newtheorem{remark}[theorem]{Remark}
\newtheorem{example}[theorem]{Example}

\newtheorem{notation}[theorem]{Notation}
\newtheorem{convention}[theorem]{Convention}

\DeclareMathOperator{\inflation}{\rightarrowtail}
\DeclareMathOperator{\deflation}{\twoheadrightarrow}

\DeclareMathOperator{\im}{im}
\DeclareMathOperator{\coim}{coim}
\DeclareMathOperator{\coker}{coker}
\DeclareMathOperator{\cone}{cone}

\DeclareMathOperator{\Mor}{Mor}
\DeclareMathOperator{\Proj}{Proj}

\DeclareMathOperator{\Hom}{Hom}

\DeclareMathOperator{\Mod}{Mod}
\DeclareMathOperator{\smod}{mod}

\DeclareMathOperator{\C}{\mathbf{C}}
\DeclareMathOperator{\D}{\mathbf{D}}
\DeclareMathOperator{\K}{\mathbf{K}}
\DeclareMathOperator{\Ac}{\mathbf{Ac}}
\DeclareMathOperator{\Acb}{\mathbf{Ac}^{\mathsf{b}}}
\DeclareMathOperator{\Acm}{\mathbf{Ac}^{-}}

\DeclareMathOperator{\Cm}{\mathbf{C}^-}
\DeclareMathOperator{\Km}{\mathbf{K}^-}
\DeclareMathOperator{\Dm}{\mathbf{D}^-}
\DeclareMathOperator{\Dstar}{\mathbf{D}^*}

\DeclareMathOperator{\Cb}{\mathbf{C}^{\mathsf{b}}}
\DeclareMathOperator{\Db}{\mathbf{D}^{\mathsf{b}}}
\DeclareMathOperator{\Dbinf}{\mathbf{D}_\infty^{\mathsf{b}}}
\DeclareMathOperator{\Kb}{\mathbf{K}^{\mathsf{b}}}

\DeclareMathOperator{\add}{add}

\newcommand{\ex}[1]{#1^{\mathsf{ex}}}

\DeclareMathOperator{\resdim}{res.dim}

\makeatletter
\newcommand{\circlearrow}{}% just in case
\DeclareRobustCommand{\circlearrow}{%
  \mathrel{\vphantom{\rightarrow}\mathpalette\circle@arrow\relax}%
}
\newcommand{\circle@arrow}[2]{%
  \m@th
  \ooalign{%
    \hidewidth$#1\circ\mkern1mu$\hidewidth\cr
    $#1\longrightarrow$\cr}%
}
\makeatother

\makeatletter
\newcommand{\myitem}[1]{%
\item[#1]\protected@edef\@currentlabel{#1}%
}
\makeatother
\begin{document}

\begin{abstract}
	It is well known that a resolving subcategory $\AA$ of an abelian subcategory $\EE$ induces several derived equivalences: the equivalence $\Dm(\AA)\to \Dm(\EE)$ holds in general and furthermore restricts to equivalence $\Db(\AA)\to \Db(\EE)$ if $\resdim_{\AA}(E)<\infty$ for any object $E\in \EE$. If the category $\EE$ is uniformly bounded, i.e.~$\resdim_{\AA}(\EE)<\infty$, one obtains a derived equivalence $\D(\AA)\to \D(\EE)$.
	
	In this paper, we show that all of the above statements hold for preresolving subcategories of (one-sided) exact categories. By passing to a one-sided language, one can remove the assumption that $\AA\subseteq \EE$ is extension-closed completely from the classical setting, yielding easier criteria and more examples. To illustrate this point, we consider the Isbell category $\II$ and show that $\II\subseteq \mathsf{Ab}$ is preresolving but $\II$ cannot be realized as an extension-closed subcategory of an exact category.
	
	We also consider a criterion given by Keller to produce derived equivalences of fully exact subcategories. We show that this criterion fits into the framework of preresolving subcategories by considering the relative weak idempotent completion of said subcategory.		
\end{abstract}

\maketitle

\tableofcontents

\input{Introduction}
\input{Preliminaries}
\input{PreResolvingSection}
\input{Examples}

\providecommand{\bysame}{\leavevmode\hbox to3em{\hrulefill}\thinspace}
\providecommand{\MR}{\relax\ifhmode\unskip\space\fi MR }
% \MRhref is called by the amsart/book/proc definition of \MR.
\providecommand{\MRhref}[2]{%
  \href{http://www.ams.org/mathscinet-getitem?mr=#1}{#2}
}
\providecommand{\href}[2]{#2}

\end{document}

%% file: Introduction.tex
\section{Introduction}

A basic technique in homological algebra is replacing an object in an abelian or exact category with a well-chosen resolution.  For example, if an abelian category $\EE$ has enough projective objects, one can replace any object $E \in \EE$ with a projective resolution $P^\bullet \to E$.  The existence of such resolutions provides a triangle equivalence $\Km(\Proj \EE) \to \Dm(\EE)$.

The notion of a resolving subcategory \cite{AuslanderBridger69} generalizes this setting: a resolving subcategory $\AA$ of an abelian or exact category $\EE$ allows one to take resolutions of all objects in $\EE$ by objects in $\AA$. For example, if $S$ is a ring, then the subcategory $\FF$ of flat $S$-modules is a resolving subcategory of $\Mod S$.  As $\FF$ is extension-closed in $\Mod S$, the category $\FF$ inherits an exact structure (so that the derived category of $\FF$ is defined, see \cite{Neeman90}) and we again find a triangle equivalence $\Dm(\FF) \to \Dm(\Mod S).$

The definition of a preresolving subcategory $\AA$ of $\EE$ is obtained from the definition of a resolving subcategory by removing the assumption that $\AA$ is extension-closed in $\EE$.  Specifically, let $\EE$ be an exact category.  A full additive subcategory $\AA$ of $\EE$ is called \emph{preresolving} if $\AA$ is closed under kernels of deflations (we say that $\AA$ lies \emph{deflation-closed} in $\EE$, see \Cref{Definition:PreResolvingAndFinitelyPreresolving}), and for every $E \in \EE$ there is a deflation $A \deflation E$ for some $A \in \AA$.

Given a preresolving subcategory $\AA \subseteq \EE$, we may attach an $\AA$-resolution dimension to each object $E \in \EE$ as the smallest numbber $n \in \bN \cup \{\infty\}$ such that there is an exact sequence $0 \to A^{-n} \to A^{-n-1} \to \cdots \to A^0 \to E$ with each $A^i \in \AA.$  We write $\resdim_\AA(\EE)$ for the supremum of $\resdim_\AA(E)$, ranging over all $E \in \EE.$

As $\AA$ need not be extension-closed in $\EE$, the category $\AA$ need not inherit an exact structure from $\EE$.  However, using the fact that $\AA$ lies deflation-closed in $\EE$, the subcategory $\AA$ still inherits a rich homological structure from $\EE$: it is a deflation-exact category (see \Cref{Proposition:DeflationClosedInheritsDeflationExactStructure} in the text).  Deflation-exact categories still admit a natural (bounded) derived category \cite{BazzoniCrivei13,HenrardvanRoosmalen19b}, and we obtain the following theorem (see \Cref{Theorem:PreResolvingDerivedEquivalence} in the text).

\begin{theorem}\label{Theorem:Introduction}
	Let $\AA$ be a preresolving subcategory of a (deflation-)exact category $\EE$. 
	\begin{enumerate}
		\item The embedding $\AA\subseteq \EE$ lifts to a triangle equivalence $\Dm(\AA)\to \Dm(\EE)$.
		\item	If $\resdim_{\AA}(E)<\infty$ for all $E\in \EE$, then the embedding lifts to a triangle equivalence $\Db(\AA)\to \Db(\EE)$.
		\item If $\resdim_{\AA}(\EE)<\infty$, then the embedding lifts to a triangle equivalence $\D(\AA)\to \D(\EE)$.
	\end{enumerate}
\end{theorem}

Similar results can be found in various places in the literature (see for example \cite[Lemma I.4.6]{Hartshorne66}, \cite[Proposition 13.2.2]{KashiwaraSchapira06}, \cite[Proposition~5.14]{Stovicek14} or \cite[Lemma 3.8]{KernerStovicekTrlifaj11} and \Cref{remark:Comparison} for a short comparison).  An example of a resolving subcategory of an abelian category $\EE$ is given by a cotilting torsion-free class $\FF$; in this case, $\resdim_{\FF}(\EE)\leq 1$.  In this setting, \Cref{Theorem:Introduction} can be recovered from \cite[Appendix~B]{BondalVandenBergh03}, generalizing the well-known tilting construction from \cite[Theorem~3.3]{HappelReitenSmalo96}.  \Cref{Theorem:Introduction} allows us to generalize this equivalence to the case where $\FF$ is a cotilting pretorsion-free class in $\EE$.  The examples in \S\ref{Section:ExampleAndApplication} are of this form.

In \Cref{Theorem:Introduction}, the preresolving subcategory $\AA$ is endowed with the following set of conflations: a sequence $A \to B \to C$ in $\AA$ is a conflation in $\AA$ if and only if it is a conflation in $\EE$.  In \cite[Theorem~12.1]{Keller96}, a variant of \Cref{Theorem:Introduction} was shown for subcategories $\AA$ which are not deflation-closed subcategories of $\EE$, and hence, not resolving.  Removing the assumption of being deflation-closed, it is not immediately clear how to obtain a meaningful deflation-exact structure on $\AA$.

In section \ref{Section:KellersCConditions}, we describe a more involved conflation structure on $\AA$ (namely, we endow $\AA$ with the largest deflation-exact structure such that the embedding $\AA \to \EE$ maps conflations to conflations).  Modifying condition \ref{C2} accordingly, we find a version of \cite[Theorem~12.1]{Keller96} for subcategories $\AA$ which are not necessarily fully exact in $\EE.$

\begin{theorem}\label{theorem:IntroductionKeller}
	Let $\EE$ be a deflation-exact category and let $\AA\subseteq \EE$ be a full additive subcategory satisfying axioms \ref{C1} and \ref{C2'}.  If we endow $\AA$ with the largest deflation-exact structure such that the embedding $\AA \to \EE$ maps conflations in $\AA$ to conflations in $\EE$, then the obvious functor $\Dm(\AA)\xrightarrow{\Phi}\Dm(\EE)$ is a triangle equivalence.

	Moreover, if $\resdim_{\AA}(E)<\infty$ for all $E\in \EE$, the functor $\Db(\AA)\to \Db(\EE)$ is an equivalence.
\end{theorem}

When $\EE$ is exact and $\AA$ is a fully exact subcategory, then \Cref{theorem:IntroductionKeller} reduces to \cite[Theorem~12.1]{Keller96}.

\paragraph*{\textbf{Acknowledgements.}} The authors wish to thank Francesco Genovese and Jan \v{S}\v{t}ov\'{\i}\v{c}ek for their useful comments and references. The second author is currently a postdoctoral researcher at FWO (12.M33.16N).

%% file: Preliminaries.tex
\section{Preliminaries}\label{section:Preliminaries}

In these preliminaries, we summarize results of \cite{BazzoniCrivei13,HenrardvanRoosmalen19b,HenrardvanRoosmalen19a} in a convenient form.

\subsection{One-sided exact categories}

\begin{definition}
	\begin{enumerate}
		\item	A \emph{conflation category} is an additive category $\EE$ together with a chosen class of kernel-cokernel pairs, called \emph{conflations}, such that this class is closed under isomorphisms. The kernel part of a conflation is called an \emph{inflation} and the cokernel part of a conflation is called a \emph{deflation}. We depict inflations by the symbol $\inflation$ and deflations by $\deflation$.
		\item An additive functor $F\colon \CC\to \DD$ between conflation categories is called \emph{(conflation-)exact} if conflations are mapped to conflations.
		\item A map $f\colon X\to Y$ in a conflation category $\EE$ is called \emph{admissible} if $f$ admits a deflation-inflation factorization, i.e.~$f$ factors as $X\deflation I\inflation Y$. An admissible morphism is depicted by $\circlearrow$.
		\item A cochain complex $(X^{\bullet}, d^{\bullet})$ in $\C(\EE)$ with $\EE$ a conflation category is called \emph{acyclic} or \emph{exact} if, for each $i \in \bZ$, the map $d_X^i\colon X^i\to X^{i+1}$ is admissible and $\ker(d_X^{i+1})\cong \im(d_X^i)$.
	\end{enumerate}
\end{definition}

\begin{remark}
	Let $\EE$ be a conflation category and let $f\colon X\to Y$ be a map. If $f$ is admissible, then the deflation-inflation factorization is unique up to isomorphism. Indeed, given a deflation-inflation factorization $X\deflation I\inflation Y$ of $f$, it is clear that $f$ admits a kernel and a cokernel, moreover, $\coker(\ker(f))=\coim(f)\cong I\cong \im(f)=\ker(\coker(f))$.
\end{remark}

\begin{definition}\label{Definition:DeflationExact}
  A \emph{deflation-exact category} $\EE$ is a conflation category satisfying the following axioms:
	\begin{enumerate}[label=\textbf{R\arabic*},start=0]
		\item\label{R0} For each $X\in \EE$, the map $X\to 0$ is a deflation.
		\item\label{R1} The composition of two deflations is a deflation.
		\item\label{R2} The pullback of a deflation along any morphism exists and deflations are stable under pullbacks.
	\end{enumerate}
	Dually, an \emph{inflation-exact category} is a conflation category $\EE$ such that the inflations satisfy the following axioms:
		\begin{enumerate}[label=\textbf{L\arabic*},start=0]
		\item\label{L0} For each $X\in \EE$, the map $0\to X$ is an inflation.
		\item\label{L1} The composition of two inflations is an inflation.
		\item\label{L2} The pushout of an inflation along any morphism exists and inflations are stable under pushouts.
	\end{enumerate}
\end{definition}

\begin{definition}\label{Definition:StronglyDeflationExact}
	Let $\EE$ be a conflation category. In addition to the properties listed in definition \ref{Definition:DeflationExact}, we also consider the following axioms:
	\begin{enumerate}[align=left]
		\myitem{\textbf{R3}}\label{R3} \hspace{0.175cm}If $i\colon A\rightarrow B$ and $p\colon B\rightarrow C$ are morphisms in $\EE$ such that $p$ has a kernel and $pi$ is a deflation, then $p$ is a deflation.
		\myitem{\textbf{R3}$^+$}\label{R3+} \hspace{0.175cm}If $i\colon A\rightarrow B$ and $p\colon B\rightarrow C$ are morphisms in $\EE$ such that $pi$ is a deflation, then $p$ is a deflation.
	\end{enumerate}
	The axioms \textbf{L3} and \textbf{L3}$^+$ are defined dually. A deflation-exact category satisfying \ref{R3} is called \emph{strongly deflation-exact}. Dually, an inflation-exact category satisfying axiom \textbf{L3} is called a strongly inflation-exact category.
\end{definition}

\begin{remark}\label{Remark:BasicDefinitions}
	\begin{enumerate}
		\item Inflation-exact categories have also been called left exact categories in \cite{BazzoniCrivei13, Rump11} and right exact categories in \cite{Rosenberg11}.  Similarly, deflation-exact categories have also been referred to as right exact or left exact categories.  As the use of left and right is not consistent, we prefer to use the above terminology.
		\item A deflation-exact category has the structure of a Grothendieck pretopology (see for example \cite{KaledinLowen15,Rosenberg11}).
		\item Axiom \ref{R0} above is axiom \textbf{R0}$^*$ in \cite{BazzoniCrivei13,HenrardvanRoosmalen19a}. This axiom ensures that all split kernel-cokernel pairs are conflations in a deflation-exact category (this is also required in \cite{Rump11}). 
		\item An exact category in the sense of Quillen (see \cite{Quillen73}) is a conflation category $\EE$ satisfying axioms \ref{R0} through \ref{R3} and \ref{L0} through \textbf{L3}. In \cite[Appendix~A]{Keller90}, Keller shows that axioms \ref{R0}, \ref{R1}, \ref{R2}, and \ref{L2} suffice to define an exact category.
		\item Axioms \ref{R3} and \textbf{L3} are sometimes referred to as Quillen's \emph{obscure axioms} (see \cite{Buhler10,ThomasonTrobaugh90}).
		\item Axiom \ref{R3+} implies axiom \ref{R3}, dually, axiom \textbf{L3}$^+$ implies axiom \textbf{L3}.
		\item For a deflation-exact category $\EE$, axiom \ref{R3+} is equivalent to $\EE$ being weakly idempotent complete and satisfying axiom \ref{R3}.
	\end{enumerate}
\end{remark}

The following proposition is used multiple times throughout this text (see \cite[Proposition 5.7]{BazzoniCrivei13} or \cite[Proposition 3.7]{HenrardvanRoosmalen19a}).

\begin{proposition}\label{Proposition:LocalizationPaperProposition3.7}
	Let $\EE$ be a deflation-exact category. For a commutative diagram
	\[\xymatrix{
		Y'\ar@{->>}[r]^{p'}\ar[d]^f & Z'\ar[d]^{g}\\
		Y\ar@{->>}[r]^p & Z
	}\] whose horizontal arrows are deflation, the following are equivalent:
	\begin{enumerate}
		\item the square is a pullback square;
		\item the square is bicartesian, i.e.~both a pullback and pushout square;
		\item the induced sequence $\xymatrix{Y'\ar@{>->}[r]^-{\begin{psmallmatrix}f\\-p'\end{psmallmatrix}} & Y\oplus Z'\ar@{->>}[r]^-{\begin{psmallmatrix}p& g\end{psmallmatrix}} & Z}$ is a conflation.
	\end{enumerate}
\end{proposition}

The following theorem (see \cite[Theorem~1.1 and Theorem~1.2]{HenrardvanRoosmalen20Obs}) highlights the importance of axioms \ref{R3} and \ref{R3+}.

\begin{theorem}\label{Theorem:ImportanceR3andR3+}
	Let $\EE$ be a deflation-exact category.
		\begin{enumerate}
			\item The category $\EE$ satisfies axiom \ref{R3} if and only if the Nine Lemma holds.
			\item The category $\EE$ satisfies axiom \ref{R3+} if and only if the Snake Lemma holds.
		\end{enumerate}
\end{theorem}

\subsection{Derived categories of one-sided exact categories}

Let $\EE$ be a one-sided exact category. The following is a small extension of \cite[Lemma~7.2]{BazzoniCrivei13}.

\begin{lemma}\label{Lemma:ConeOfAcyclic}
	The mapping cone $\cone(f^{\bullet})$ of a map $f^{\bullet}\colon X^{\bullet}\to Y^{\bullet}$ between acyclic complexes $X^{\bullet},Y^{\bullet}\in \C(\EE)$ is acyclic.  Furthermore, for each $n \in \bZ$, there is are conflations:
	\[\mbox{$\im(d^{n-1}_X) \inflation X^n \oplus \im(d^{n-1}_X) \deflation \im(d^{n-1}_{\cone(f)})$ and $\im(d^{n-1}_{\cone(f)}) \inflation Y^{n} \oplus \im(d^{n}_X) \deflation \im(d^{n}_Y).$}\]
\end{lemma}

\begin{proof}
The proof is as in \cite[Lemma~7.2]{BazzoniCrivei13}, where it was noted that, for each $n \in \bZ$, there is a commutative diagram:
\[\xymatrix{ {\im(d^{n-1}_X)} \ar[d] \ar@{>->}[r]\ar@{}[rd]|{\text{BC}} & X^n \ar@{->>}[r] \ar[d] & {\im(d^{n}_X)} \ar@{=}[d] \\
{\im(d^{n-1}_Y)} \ar@{=}[d] \ar@{>->}[r] & {\im(d^{n-1}_{\cone(f)})} \ar@{->>}[r] \ar[d]\ar@{}[rd]|{\text{BC}} & {\im(d^{n}_Y)} \ar[d] \\
{\im(d^{n-1}_Y)} \ar@{>->}[r] & Y^n \ar@{->>}[r] & {\im(d^{n}_Y)}}\]
where the rows are conflations and the squares labeled BC are bicartesian.  The required conflations are obtained by \cite[Proposition 3.7]{HenrardvanRoosmalen19a}.
\end{proof}

Write $\Ac(\EE)\subseteq \C(\EE)$ for the unbounded acyclic complexes. Analogous to exact categories \cite{Neeman90}, one can define the unbounded derived category $\D(\EE)$ as the Verdier localization $\K(\EE)/\left\langle\Ac(\EE)\right\rangle_{\text{thick}}$ of the homotopy category $\K(\EE)$ by the thick closure of the triangulated subcategory of acyclic complexes. Note that $\D(\EE)$ is the localization of $\K(\EE)$ with respect to the saturated multiplicative system $\SS$ of those maps $f^{\bullet}\in \Mor(\K(\EE))$ such that $\cone(f^{\bullet})\in \left\langle\Ac(\EE)\right\rangle_{\text{thick}}$. The right bounded derived category $\Dm(\EE)$ and bounded derived category $\Db(\EE)$ can be defined analogously.

\begin{convention}\label{Convention:QuasiIsomorphism}
	We use the following convention: a map $f^{\bullet}\in \Mor(\C(\EE))$ such that $\cone(f^{\bullet})\in \Ac(\EE)$ is called a \emph{quasi-isomorphism}. Note that $\SS$ is simply the multiplicative system in $\K(\EE)$ obtained as the saturated closure of the quasi-isomorphisms, in particular, the derived category $\D(\EE)$ is equivalent to the localization $\K(\EE)[\SS^{-1}]$.
\end{convention}

Following the above convention, a composition of quasi-isomorphisms need not be a quasi-isomorphism but it does produce a morphism in $\SS\subseteq \Mor(\K(\EE))$.

The following proposition summarizes some useful properties of the derived category (see \cite{HenrardvanRoosmalen19b}).

\begin{proposition}\label{Proposition:BasicPropertiesDerivedCategory}
	Let $\EE$ be a deflation-exact category and let $*\in \{-,b\}$
	\begin{enumerate}
		\item The natural embedding $i\colon \EE\to \D(\EE)$ is fully faithful.
		\item For all $X, Y\in \EE$ and $n>0$, $\Hom_{\D(\EE)}(\Sigma^n iX, iY)=0$.
		\item Every conflation $X\inflation Y\deflation Z$ in $\EE$ maps to a triangle $iX\to iY\to iZ\to \Sigma iX$ in $\D(\EE)$.
		\item\label{item:R3Triangle} Axiom \ref{R3} holds if and only if every sequence $X\to Y\to Z$ that lifts to a triangle in $\D(\EE)$ is a conflation in $\EE$.
		\item The category $\Ac^*(\EE)$ is a thick triangulated subcategory of $\K^*(\EE)$ if and only if $\EE$ satisfies axiom \ref{R3+}.
		\item The category $\Ac(\EE)$ is a thick triangulated subcategory of $\K(\EE)$ if and only if $\EE$ is idempotent complete and satisfies axiom \ref{R3}.
	\end{enumerate}
\end{proposition}

%% file: PreResolvingSection.tex
\section{Preresolving subcategories and bounded derived equivalences}\label{section:Preresolving}

Throughout this section, $\EE$ denotes a deflation-exact category (all results can be dualized to inflation-exact categories, see \Cref{Remark:DualSetting}).

\subsection{Definitions and basic properties}

We start by recalling the notion of a preresolving subcategory and the resolution dimension (see \cite[Definition~2.1]{Stovicek14} or \cite[Definition~2]{Rump20}).

\begin{definition}\label{Definition:PreResolvingAndFinitelyPreresolving}
	A full additive subcategory $\AA\subseteq \EE$ is called \emph{preresolving} if the following two conditions are met:
	\begin{enumerate}[label=\textbf{PR\arabic*},start=1]
		\item\label{PR2} For every $E\in \EE$ there exists a deflation $A\deflation E$ with $A\in \AA$.
		\item\label{PR1} The subcategory $\AA\subseteq \EE$ is \emph{deflation-closed}, i.e.~for every conflation $X \inflation Y \deflation Z$, if $Y,Z \in \AA$, then $X \in \AA$.
	\end{enumerate}
	
	Let $\AA\subseteq \EE$ be a an additive subcategory satisfying axiom \ref{PR2}. The \emph{resolution dimension} of an object $E\in \EE$ with respect to $\AA$, denoted $\resdim_{\AA}(E)$, is the smallest integer $n\geq 0$ such that there exists an exact sequence
	\[0 \to A^{-n} \to A^{-n+1} \to \cdots \to A^{0} \to E \to 0\]
in $\EE$ with each $A^i \in \AA$.  If such an $n$ does not exist, we write $\resdim_\AA(E) = \infty$.

Furthermore, we define $\resdim_{\AA}(\EE)= \sup \{\resdim_{\AA}(E) \mid \forall E \in \EE\} \in \mathbb{Z}^+\cup \{\infty\}$.

A preresolving subcategory $\AA\subseteq \EE$ is called \emph{finitely} preresolving if for each object $E\in \EE$, $\resdim_{\AA}(E)<\infty$, and is called a \emph{uniformly} preresolving subcategory if $\resdim_{\AA}(\EE)<\infty$. Furthermore, a preresolving subcategory $\AA\subseteq \EE$ is called \emph{resolving} if $\AA\subseteq \EE$ is extension-closed.
\end{definition}

\begin{remark}
It follows from axiom \ref{PR2} that each object in $\EE$ has a resolution by objects in $\AA$.
\end{remark}

\begin{notation}
For an object $E \in \EE$, we often write $A^\bullet \to E$ for the complex $\cdots \to A^{-n} \to A^{-n+1} \to \cdots \to A^{0} \to E \to 0 \to \cdots$  Here, we assume that $E$ has degree $1$.
\end{notation}

\begin{remark}\label{Remark:DualSetting}
The notions of a coresolving and a precoresolving subcategory of an exact or inflation-exact category are dual.
\end{remark}

\begin{remark}
	The category $\PP$ of projectives of \cite[Example~4.10]{BergvanRoosmalen14} is a finitely resolving subcategory of $\smod(\fb)$, but not uniformly resolving.
\end{remark}

\begin{proposition}\label{Proposition:DeflationClosedInheritsDeflationExactStructure}
	Let $\AA\subseteq \EE$ be a deflation-closed subcategory. The category $\AA$ inherits a deflation-exact structure from $\EE$ (the conflations are simply those conflations in $\EE$ that lie in $\AA$). Moreover, if $\EE$ satisfies axiom \ref{R3+} or \ref{R3}, then $\AA$ satisfies axiom \ref{R3+} or \ref{R3}, respectively.
\end{proposition}

\begin{proof}
	Axioms \ref{R0} and \ref{R1} follow directly from the definition. To show axiom \ref{R2}, let $X\inflation Y\deflation Z$ be a conflation in $\AA$ and let $A\to Z$ be any map in $\AA$. Applying axiom \ref{R2} in $\EE$, we obtain the following commutative diagram:
	\[\xymatrix{
		X\ar@{>->}[r]\ar@{=}[d] & P\ar@{->>}[r]\ar[d] & A\ar[d]\\
		X\ar@{>->}[r] & Y\ar@{->>}[r] & Z
	}\] By \Cref{Proposition:LocalizationPaperProposition3.7}, the induced kernel-cokernel pair $P\inflation A\oplus Y\deflation Z$ is a conflation in $\EE$. As $\AA\subseteq \EE$ is deflation-closed, it follows that $P\in \AA$ as required.
	
	That axiom \ref{R3+} is preserved is straightforward to show. That axiom \ref{R3} is preserved is shown in \cite[Proposition~3.11]{HenrardvanRoosmalen19a}.
\end{proof}

If $\AA\subseteq \EE$ is deflation-closed, \Cref{Proposition:DeflationClosedInheritsDeflationExactStructure} endows $\AA$ with a natural deflation-exact structure; in this case, the derived category $\D(\AA)$ has been defined in \cite{BazzoniCrivei13}. It is clear that $\Ac(\AA)\subseteq \Ac(\EE)$; however, in general, $\Ac(\AA)\neq \C(\AA)\cap \Ac(\EE)$. Thus, a complex $A^{\bullet}\in \C(\AA)$ which is acyclic in $\EE$ need not be acyclic in $\AA$, see \Cref{Remark:UnboundedAcyclicIsAnnoying} below. For this reason, we make the following distinction.
                   
\begin{definition}
	As in \Cref{Convention:QuasiIsomorphism}, let $\SS$ denote the collection of morphisms $f^{\bullet}\in \Hom_{\C(\EE)}(X^{\bullet},Y^{\bullet})$ such that $\cone(f^{\bullet})\in \Ac(\EE)$. If $\AA\subseteq \EE$ is a deflation-closed subcategory, we write $\TT$ for those morphisms $f^{\bullet}\in \Hom_{\C(\AA)}(A^{\bullet},B^{\bullet})$ such that $\cone(f^{\bullet})\in \Ac(\AA)$. The bounded variants $\TT^-,\TT^\mathsf{b},\SS^-$ and $\SS^\mathsf{b}$ are defined analogously.
\end{definition}

\begin{remark}
If $\AA\subseteq \EE$ is deflation-closed, and hence deflation-exact by \Cref{Proposition:DeflationClosedInheritsDeflationExactStructure}, then $\D(\AA)=\K(\AA)[\TT^{-1}]$.
\end{remark}

\begin{proposition}\label{Proposition:AcyclicComplexesOfDeflationClosedSubcategoryCoincide}
	Let $\AA\subseteq \EE$ be a deflation-closed subcategory. Then $\Acm(\EE)\cap\Cm(\AA)=\Acm(\AA)$. In particular, $\TT^{-}=\SS^{-}\cap\Mor(\Km(\AA))$ and $\TT^\mathsf{b}=\SS^{\mathsf{b}}\cap \Mor(\Kb(\AA))$.
\end{proposition}

\begin{proof}
	Let $A^{\bullet}\in \Cm(\AA)$. Without loss of generality, assume that $A^{k}=0$ for all $k>0$. It suffices to show that $A^{\bullet}$ is acyclic in $\AA$ if and only if $A^{\bullet}$ is acyclic in $\EE$. 
	
	If $A^{\bullet} \in \Ac(\AA)$, then we have that $A^\bullet \in \Ac(\EE)$ (as \Cref{Proposition:DeflationClosedInheritsDeflationExactStructure} yields that the embedding $\AA\to \EE$ is conflation-exact).  Conversely, assume that $A^{\bullet} \in \Ac(\EE)$. In this case, $d_A^{-1}\colon A^{-1}\deflation A^{0}$ is a deflation and hence its kernel $K^{-1}\inflation A^{-1}$ belongs to $\AA$ as $\AA\subseteq \EE$ is deflation-closed. Moreover, since the complex $A^\bullet$ is acyclic, we have $K^{-1}\cong \im(d_A^{-2})$. Repeating this argument on the deflation $A^{-2}\deflation \im(d_A^{-2})$, one finds that $\im(d_A^{-3})\in \AA$.  Proceeding in this fashion yields that $A^{\bullet}$ is acyclic in $\AA$.
\end{proof}

\begin{remark}\label{Remark:UnboundedAcyclicIsAnnoying}
	\Cref{Proposition:AcyclicComplexesOfDeflationClosedSubcategoryCoincide} does not need to hold for unbounded complexes. Indeed, consider the category $\EE=\smod(A)$ of finitely generated $k$-modules over the $k$-algebra $A=k[X]/(X^2)$ and let $\AA = \add(\EE)$ be the full additive subcategory generated by $A$. The complex 
	\[\dots \xrightarrow{\cdot X} A\xrightarrow{\cdot X} A \xrightarrow{\cdot X} A \xrightarrow{\cdot X} \dots\]
	belongs to $\AA$ and is acyclic in $\EE$. Clearly, the images of the differentials do not belong to $\AA$ and thus the above sequence is not acyclic in $\AA$.
\end{remark}

\subsection{Induced bounded derived equivalences}

The goal of this section is to prove the following theorem. 

\begin{theorem}\label{Theorem:PreResolvingDerivedEquivalence}
	Let $\AA\subseteq \EE$ be a full additive subcategory and write $\phi\colon \AA\hookrightarrow \EE$ for the embedding. 
	\begin{enumerate}
		\item If $\AA\subseteq \EE$ is preresolving, the functor $\phi$ lifts to a triangle equivalence $\Phi\colon \Dm(\AA)\to \Dm(\EE)$.
		\item	If $\AA\subseteq \EE$ is finitely preresolving, the functor $\phi$ lifts to a triangle equivalence $\Phi\colon \Db(\AA)\to \Db(\EE)$.
	\end{enumerate}
\end{theorem}

That the above functors are essentially surjective is settled by the following lemma.

\begin{lemma}\label{Lemma:ReplacementTechnique}
	Let $\AA\subseteq \EE$ be a full additive subcategory satisfying axiom \ref{PR2}.
	\begin{enumerate}
		\item For each $E^{\bullet}\in \Cm(\EE)$, there exists a complex $A^{\bullet}\in \Cm(\AA)$ and a map $f^{\bullet}\colon A^{\bullet}\to E^{\bullet}$ in $\Cm(\EE)$ such that $f^{\bullet}\in \SS^{-}$.
		\item Assume that $\resdim_{\AA}(E)<\infty$ for all $E\in \EE$.  For each $E^{\bullet}\in \Cb(\EE)$, there exists a complex $A^{\bullet}\in \Cb(\AA)$ and a map $f^{\bullet}\colon A^{\bullet}\to E^{\bullet}$ in $\Cb(\EE)$ such that $f^{\bullet}\in \SS^\mathsf{b}$.
	\end{enumerate}
\end{lemma}

\begin{proof}
	We only show the right bounded case as the bounded case is similar.
	
	Let $E\in \Cm(\EE)$. Without loss of generality we assume that $E^{i}=0$ for all $i>1$. By axiom \ref{PR2}, there exists a deflation $A\deflation E^0$ with $A\in \AA$. By axiom \ref{R2}, the pullback $P^{-1}$ of $d_E^{-1}\colon E^{-1}\to E^0$ along $A\deflation E^0$ exists. By the pullback property, we obtain the following commutative diagram:
	\[\xymatrix{
		A_1^{\bullet}\ar[d]^{f_1^{\bullet}} && \dots \ar[r]&E^{-2}\ar[r]\ar@{=}[d]&P^{-1}\ar[r]\ar@{->>}[d]&A^0\ar@{->>}[d]\ar[r]&0\ar[d]\ar[r]&\dots\\
		E^{\bullet} && \dots \ar[r]&E^{-2}\ar[r]&E^{-1}\ar[r]&E^0\ar[r]&0\ar[r]&\dots\\
	}\]
	By \Cref{Proposition:LocalizationPaperProposition3.7}, $\cone(f_1^{\bullet})\in \Ac(\EE)$ and thus $f_1^{\bullet}$ is a quasi-isomorphism. Similarly, by axiom \ref{PR2}, there is a deflation $A^{-1}\deflation P^{-1}$. Taking the pullback $P^{-2}$ along $E^{-2}\to P^{-1}$, we obtain the following commutative diagram:
		\[\xymatrix{
		A_2^{\bullet}\ar[d]^{f_2^{\bullet}} && \dots \ar[r]&E^{-3}\ar[r]\ar@{=}[d]&P^{-2}\ar[r]\ar@{->>}[d]&A^{-1}\ar@{->>}[d]\ar[r]&A^{0}\ar@{=}[d]\ar[r]&0\ar[r]&\dots\\
		A_1^{\bullet} && \dots \ar[r]&E^{-3}\ar[r]&E^{-2}\ar[r]&P^{-1}\ar[r]&A^0\ar[r]&0\ar[r]&\dots\\
	}\] As above, $f_2^{\bullet}$ is a quasi-isomorphism. Iterating this procedure, one constructs a complex $A^{\bullet}$ from right to left; moreover, the natural map $f^{\bullet}\colon A^{\bullet}\to E^{\bullet}$ belongs to $\SS^-$ as one can check degreewise.	
\end{proof}

We are now in a position to prove \Cref{Theorem:PreResolvingDerivedEquivalence}.

\begin{proof}[Proof of \Cref{Theorem:PreResolvingDerivedEquivalence}]
	We only prove the first statement, the second is similar.
	
	As $\AA\subseteq \EE$ is deflation-closed, $\AA$ inherits a deflation-exact structure from $\EE$ by \Cref{Proposition:DeflationClosedInheritsDeflationExactStructure}. Hence $\phi\colon \AA \to \EE$ is an exact functor and thus lifts to a triangle functor $\Phi\colon \Dm(\AA)\to \Dm(\EE)$. 
	
	That $\Phi$ is essentially surjective follows from \Cref{Lemma:ReplacementTechnique}. By \Cref{Lemma:ReplacementTechnique} and \Cref{Proposition:AcyclicComplexesOfDeflationClosedSubcategoryCoincide}, the conditions of (the dual of) \cite[Proposition~7.2.1 (ii)]{KashiwaraSchapira06} are satisfied and thus $\Phi$ is fully faithful.
\end{proof}

\begin{remark}
	That $\Phi$ is fully faithful also follows from \cite[Proposition~10.2.7]{KashiwaraSchapira06} or \cite[Lemma in section 1.6]{Positselski11}. 
\end{remark}

\begin{remark}\label{remark:Comparison}
	\Cref{Theorem:PreResolvingDerivedEquivalence} is akin to \cite[Lemma~3.8]{KernerStovicekTrlifaj11} (or \cite[Proposition~13.2.2]{KashiwaraSchapira06}), i.e.~given an abelian category $\EE$ and a full additive subcategory $\AA\subseteq \EE$ such that every object of $\EE$ admits a finite $\AA$-resolution, the embedding lifts to a triangle equivalence $\Db(\AA;\EE)\to \Db(\EE)$. Here $\Db(\AA;\EE)$ is the derived category of $\AA$ relative to $\EE$ (see \cite[Definition~3.6]{KernerStovicekTrlifaj11}), that is, $\Db(\AA;\EE) \coloneqq \Kb(\AA)/\Acb(\AA;\EE)$ where $\Acb(\AA;\EE) \coloneqq \Kb(\AA)\cap \Acb(\EE)$ are the bounded cochain complexes in $\AA$ which are acyclic in $\EE$. If $\AA\subseteq \EE$ is finitely preresolving, \Cref{Proposition:AcyclicComplexesOfDeflationClosedSubcategoryCoincide} yields that $\Db(\AA;\EE)\simeq \Db(\AA)$ are triangle equivalent. As such, we recover \cite[Lemma~3.8]{KernerStovicekTrlifaj11}.
\end{remark}

\begin{remark}
The bounded derived $\infty$-category of a one-sided exact category is defined in \cite{HenrardvanRoosmalen19b}, following \cite{AntieauGepnerHeller19}.  Given an preresolving subcategory $\AA$ of a (one-sided) exact category $\EE$, one obtains a functor $\Phi_\infty\colon \Dbinf(\AA) \to \Dbinf(\EE)$ of stable $\infty$-categories.  If $\AA$ is finitely preresolving in $\EE$, then $\Phi_\infty$ is an equivalence (this follows \Cref{Theorem:PreResolvingDerivedEquivalence}, using \cite[Theorem I.3.3]{NikolausScholze18} or \cite[Proposition 5.14]{BlumbergGepnerTabuada13}).  In particular, for any additive or localizing invariant $\bK$ (in the sense of \cite{BlumbergGepnerTabuada13}, such as non-connective $K$-theory) we have $\bK(\Dbinf(\AA)) \cong \bK(\Dbinf(\EE)).$
\end{remark}

\section{The resolution dimension and the unbounded derived equivalence}

Throughout this section, $\EE$ denotes a deflation-exact category. The goal of this section is to prove \Cref{Theorem:UnboundedMainTheorem} below, which is the unbounded analogue of \Cref{Theorem:PreResolvingDerivedEquivalence}. The strategy of the proof is identical to the proof of \Cref{Theorem:PreResolvingDerivedEquivalence}. \Cref{Lemma:ReplacementTechnique} directly generalizes to the unbounded setting with the additional assumption that $\resdim_{\AA}(\EE)<\infty$ (see \Cref{Lemma:UnboundedEssentialSurjectivity} below). However, \Cref{Proposition:AcyclicComplexesOfDeflationClosedSubcategoryCoincide} does not directly carry over to the unbounded setting. Nonetheless, \Cref{Corollary:DetectingAcyclicComplexesWithTermsInAAInD(A)} below provides a suitable generalization such that the proof of \Cref{Theorem:PreResolvingDerivedEquivalence} carries over verbatim.

\begin{theorem}\label{Theorem:UnboundedMainTheorem}
	Let $\AA\subseteq \EE$ be a uniformly preresolving subcategory. The embedding $\phi\colon \AA\to \EE$ lifts to a triangle equivalence $\D(\AA)\to \D(\EE)$.
\end{theorem}

The embedding $\AA \to \EE$ lifts to a triangle functor $\D(\AA)\to \D(\EE)$ by \Cref{Proposition:DeflationClosedInheritsDeflationExactStructure}.  The first step of our proof is to establish that $\Phi$ is essentially surjective.  This will be done in \Cref{Lemma:UnboundedEssentialSurjectivity} below.  The proof is an adaptation of \cite[Lemma I.4.6]{Hartshorne66} to the (one-sided) exact setting.% The proof is akin to \Cref{Lemma:ReplacementTechnique} but exploits the additional assumption that $\resdim(\EE)<\infty$.

\begin{lemma}\label{Lemma:UnboundedEssentialSurjectivity}
	Let $\AA\subseteq \EE$ be a preresolving subcategory with $\resdim_\AA (\EE) = n \in \bN$. Let $E^{\bullet}\in \C(\EE)$. There exists a map $f^{\bullet}\colon A^{\bullet}\to E^{\bullet}$ in $\K(\EE)$ such that $A^{\bullet}\in \C(\AA)$ and $f^{\bullet}\in \SS$.
\end{lemma}

\begin{proof}
	Let $E^{\bullet}\in \C(\EE)$. From \Cref{Lemma:ReplacementTechnique}, we obtain the following commutative diagram
		\[\xymatrix{
			A_1^{\bullet}\ar[d]^{f_1^{\bullet}}&& \dots\ar[r] & A^{-2}\ar[r]\ar[d] & A^{-1}\ar[r]\ar[d] & A^{0}\ar[r]\ar[d] & E^{1}\ar[r]\ar@{=}[d] & E^2\ar[r]\ar@{=}[d] & \dots\\
			E^{\bullet}&& \dots\ar[r] & E^{-2}\ar[r] & E^{-1}\ar[r] & E^{0}\ar[r] & E^{1}\ar[r] & E^2\ar[r] & \dots\\
		}\] such that $A^i\in \AA$ and $f_1^{\bullet}$ is a quasi-isomorphism.  We provide an $n$-step procedure to construct a morphism $h^\bullet\colon A^\bullet_2 \to A^\bullet_1$ in $S$ satisfying the following properties: $A_2^i \in \AA$ for $i < 2$, and $h^i\colon A_1^i \to A_2^i$ is the identity whenever $i \not\in \{ -n, \ldots, 0, 1\}$. 
		
		Let $C^{\bullet}\in \Cb(\AA)$ be a finite $\AA$-resolution of $E^1$ exhibiting that $\resdim(E)\leq n$. Consider the following commutative diagram:
		\[\xymatrix{
			&&&&&\im(d_C^{-1})\ar@{>->}[d]&&&\\
			\dots\ar[r] & A^{-2}\ar[r]\ar@{=}[d] & B^{-1}\ar[r]\ar@{->>}[d]\ar@{}[rd]|{\text{PB}} & B^0\ar@{->>}[r]\ar@{->>}[d] & P\ar[r]\ar@{->>}[d]\ar@{}[rd]|{\text{PB}} & C^0\ar[r]\ar@{->>}[d] & E^2\ar[r]\ar@{=}[d] & E^3\ar[r]\ar@{=}[d] & \dots\\
			\dots\ar[r] & A^{-2}\ar[r] & A^{-1}\ar[r] & A^0\ar@{=}[r] & A^0\ar[r] & E^1\ar[r] & E^2\ar[r] & E^3\ar[r] & \dots
		}\] Here, the above diagram is constructed from right to left as follows: the squares labeled PB are pullback squares, which exist by axiom \ref{R2}; the deflation $B^0\deflation P$ (with $B^0\in \AA$) exists by axiom \ref{PR2}. Note that $B^{-1}\in \AA$ as it fits into the conflation $B^{-1}\inflation A^{-1}\oplus B^0\deflation A^0$ and $\AA\subseteq \EE$ is deflation-closed.
		
		Consider now the following commutative diagram:
		\[\xymatrix{
			A_{1,-1}^{\bullet}\ar[d]^{\text{q.i.}}&&\dots\ar[r] & A^{-2}\ar@{=}[d]\ar[r] & B^{-1}\ar[r]\ar@{=}[d] & \im(d_C^{-1})\oplus C^0\ar[r]\ar@{->>}[d] & C^0\ar[r]\ar@{->>}[d] & E^2\ar[r]\ar@{=}[d] & \dots\\
			\ar[d]^{\text{q.i.}}&&\dots\ar[r] & A^{-2}\ar[r]\ar@{=}[d] & B^{-1}\ar[r]\ar@{->>}[d] & B^0\ar[r]\ar@{->>}[d]\ar@{.>}[ru]|{\text{PB}} & E^1\ar[r]\ar@{=}[d]  & E^2\ar[r]\ar@{=}[d]  & \dots\\
			A_1^{\bullet}&&\dots\ar[r] & A^{-2\ar[r]} & A^{-1}\ar[r]\ar@{}[ru]|{\text{PB}} & A^0\ar[r] & E^1\ar[r] & E^2\ar[r] & \dots\\
		}\] Here, the top pullback square is a split pullback as $B^0\to E^1$ factors through $C^0\deflation E^1$. Note that all terms up to and including degree 1 of the complex $A_{1,-1}^{\bullet}$ belong to $\AA$, with the possible exception of $\im(d_C^{-1})\oplus C^0$. By construction, $\resdim(\im(d_C^{-1})\oplus C^0)<\resdim(E^1)\leq n$.  Hence, iterating the above procedure (the next step is to now taking a strictly shorter resolution of $\im(d_C^{-1})\oplus C^0$), one obtains a complex $A_2^{\bullet}=A_{1,-n}^{\bullet}$ such that all terms belong to $\AA$ up to degree 1 and $A_2^{\bullet}$ maps to $A_1^{\bullet}$ via a map in $\SS$. Furthermore, $A_2^{-n-l}=A_1^{-n-l}$ for all $l\geq 1$, $A_2^{1}=C^0$ and $A_2^{k}=E^{k}$ for all $k\geq 2$. Iterating this entire procedure, one produces a string of morphisms $E^{\bullet}\leftarrow A_1^{\bullet} \leftarrow A_2^{\bullet} \leftarrow \dots$ in $\SS$ that stabilizes in each degree. The result follows.
\end{proof}

In order to prove \Cref{Theorem:UnboundedMainTheorem}, it suffices the show that $\D(\AA)\to \D(\EE)$ is fully faithful. The following proposition already indicates that in many cases, fully faithfulness readily holds.

\begin{proposition}\label{Proposition:UnboundedTheoremWithAdditionalAssumptions}
	Let $\AA\subseteq \EE$ be a uniformly preresolving subcategory.
	\begin{enumerate}
		\item If $\AA\subseteq \EE$ is closed under normal subobjects, i.e.~$X\in \AA$ if $X\inflation A$ is an inflation with $A\in \AA$, then the natural map $\D(\AA)\to \D(\EE)$ is a triangle equivalence.
		\item If $\AA\subseteq \EE$ is extension-closed and $\EE$ satisfies axiom \ref{R3}, then the natural map $\D(\AA)\to \D(\EE)$ is a triangle equivalence.
	\end{enumerate}
\end{proposition}

\begin{proof}
	As $\D(\AA)\to \D(\EE)$ is already essentially surjective by \Cref{Lemma:UnboundedEssentialSurjectivity}, it suffices to show that if $A^{\bullet}\in \C(\AA)\cap \Ac(\EE)$, then $A^{\bullet}\in \Ac(\AA)$. We show this property in both cases.
	\begin{enumerate}
		\item Assume that $\AA\subseteq \EE$ is closed under normal subobjects. Let $A^{\bullet}\in \C(\AA)\cap \Ac(\EE)$.  Clearly $\im(d_A^{i})\in \AA$ for all $i\in \mathbb{Z}$ as $\im(d_A^{i})$ is a normal subobject of $A^{i+1}$. The results follows.
		\item	Assume that $\AA\subseteq \EE$ is extension-closed and $\EE$ satisfies axiom \ref{R3}. The reader may verify that the proof of \cite[Proposition~2.3]{Stovicek14} only uses that $\AA\subseteq \EE$ is extension-closed and the Nine Lemma. By \Cref{Theorem:ImportanceR3andR3+}, the Nine Lemma holds in $\EE$ as $\EE$ satisfies axiom \ref{R3}. Now let $A^{\bullet}\in \C(\AA)\cap \Ac(\EE)$ and consider $\im(d_A^{i})\in \EE$ for some $i\in \mathbb{Z}$. The following sequence yields a resolution of $\im(d_A^{i})$:
		\[\im(d_A^{i-n-1})\to A^{i-n}\dots \to A^{i-1} \to A^{i}\deflation \im(d_A^{i})\to 0.\]
		As $\resdim_\AA(\EE)=n$, \cite[Proposition~2.3]{Stovicek14} yields that $\im(d_A^{i-n-1})\in \AA$. As $i$ is chosen arbitrarily, $A^{\bullet}\in \Ac(\AA)$. The result follows. \qedhere
	\end{enumerate}
\end{proof}

In order to prove \Cref{Theorem:UnboundedMainTheorem} without the additional assumptions from \Cref{Proposition:UnboundedTheoremWithAdditionalAssumptions}, more work is needed. We start with the following lemma.

\begin{lemma}\label{lemma:ResolutionLength}
	Let $\AA$ be a preresolving subcategory in $\EE$.  Let $A^\bullet \to E$ and $B^\bullet \to E$ be resolutions of $E\in \EE$ such that there is a commutative diagram
\[\xymatrix{B^\bullet \ar[r] \ar[d]^{f^{\bullet}} & E \ar@{=}[d] \\
A^\bullet \ar[r] & E}\]
If $\ker d_A^{k} \in \AA$ for some $k \leq 0$, then $\ker d_B^k \in \AA$ (where $d_A^0\colon A^0\deflation E$). Moreover, $\cone(f^{\bullet})\in \Acm(\AA)$.
\end{lemma}             

\begin{proof}
	Assume that $\ker(d_A^{k})\in \AA$. Replacing the complex $A^{\bullet}$ by the truncated complex $\tau^{\leq k-1}A^{\bullet}$ with $\ker(d_A^{k})$ in degree $k-1$ and taking the cone of the above diagram, we obtain the following complex:
	\[\dots\to B^{k-2} \to B^{k-1} \xrightarrow{\begin{psmallmatrix}-d_B^{k-1}\\ \alpha \end{psmallmatrix}} B^k\oplus \ker(d_A^k) \to B^{k+1}\oplus A^k\to \dots \to B^0\oplus A^{-1}\to E\oplus A^0\to E \to 0 \to \dots\]
	By \Cref{Lemma:ConeOfAcyclic}, the above complex is acyclic. It follows that $\coker(-d_B^{k-2})\cong \ker(-d_B^{k})$ is isomorphic to the image of the admissible map $B^{k-1} \xrightarrow{\begin{psmallmatrix}-d_B^{k-1}\\ \alpha \end{psmallmatrix}} B^k\oplus \ker(d_A^k)$. One readily verifies that removing the contractible $E\xrightarrow{1_E} E$ from the above acyclic complex yields $\cone(f^{\bullet})\in \Acm(\EE)\cap \Cm(\AA)$. \Cref{Proposition:AcyclicComplexesOfDeflationClosedSubcategoryCoincide} now implies that $\cone(f^{\bullet})\in \Acm(\AA)$. It follows that $\ker(d_B^{k})\in \AA$.
\end{proof}

\begin{proposition}\label{proposition:ResolutionLength}
	Let $\AA$ be a preresolving subcategory in $\EE.$  Let $E \in \EE$ with $\resdim_\AA(E) = n < \infty$.  For any $\AA$-resolution $A^\bullet \to E$, there is a commutative diagram
\[\xymatrix{B^\bullet \ar[r] \ar[d]^{f^{\bullet}} & E \ar@{=}[d] \\
A^\bullet \ar[r] & E.}\]
where the top row is an $\AA$-resolution and $\ker d_B^{-n+1} \in \AA$. Moreover, $\cone(f^{\bullet})\in \Acm(\AA)$.
\end{proposition}

\begin{proof}
	Let $C^\bullet \to E$ be an $\AA$-resolution of $E$ exhibiting that $\resdim_\AA E = n$ (so $C^{-n-1} = 0$ and hence $\ker(d^{-n+1}) = C^{-n} \in \AA$). By axiom \ref{R2}, we obtain the pullback $P^0$ of $A^0\deflation E$ along $C^0\deflation E$. By axiom \ref{PR2}, there is a deflation $B^0\deflation P^0$ with $B^0\in \AA$. By axiom \ref{R1}, we obtain the following commutative diagram
	\[\xymatrix{
		\im(d_C^{-1})\ar@{>->}[r] & C^0\ar@{->>}[r] & E\\
		I^{-1}_B\ar@{>->}[r]\ar@{.>>}[u]\ar@{.>>}[d] & B^0\ar@{->>}[r]\ar@{->>}[u]\ar@{->>}[d] & E\ar@{=}[u]\ar@{=}[d]\\
		\im(d_A^{-1})\ar@{>->}[r] & A^0\ar@{->>}[r] & E
	}\] whose rows are conflations. Note that the left squares are pullback squares and, by axiom \ref{R2}, the induced vertical maps are deflations.
	
	We now iteratively construct $B^{\bullet}$ as follows: Applying axiom \ref{R2} thrice, we obtain the following commutative diagram
	\[\xymatrix@R=1pc@C=1pc{
		&C^{-1}\ar@{->>}[rr]&&\im(d_C^{-1})\\
		&&Q^{-1}\ar[lu]\ar@{->>}[rd]\ar@{}[ru]|{\text{PB}}&\\
		B^{-1}\ar@{->>}[r]&R^{-1}\ar@{->>}[ru]\ar@{->>}[rd]\ar@{}[rr]|{\text{PB}}&&I_B^{-1}\ar[uu]\ar[dd]\\
		&&P^{-1}\ar[ld]\ar@{->>}[ru]\ar@{}[rd]|{\text{PB}}&\\
		&A^{-1}\ar@{->>}[rr]&&\im(d_A^{-1})\\
	}\] where the deflation $B^{-1}\deflation R^{-1}$ is obtained by axiom \ref{PR2}. By axiom \ref{R1}, one obtains the conflation $I^{-2}_B\inflation B^{-1}\deflation I^{-1}_B$ and the induced maps $\im(d_C^{-2})\leftarrow I^{-2}_B\rightarrow \im(d_A^{-2})$. Iterating this procedure, one obtains morphisms $A^{\bullet} \stackrel{f^{\bullet}}{\leftarrow} B^{\bullet} \rightarrow C^{\bullet}$ such that $B^{\bullet}$ is an $\AA$-resolution of $E$. By \Cref{lemma:ResolutionLength}, $\ker(d_B^{-n+1})\in \AA$. That $f^{\bullet}$ is a quasi-isomorphism in $\C(\AA)$ is shown in the proof of \Cref{lemma:ResolutionLength}.
\end{proof} 

\begin{corollary}\label{Corollary:DetectingAcyclicComplexesWithTermsInAAInD(A)}
	Let $\AA$ be a preresolving subcategory in $\EE$ with $\resdim_\AA(\EE) = n \in \bN$.  For each $E^{\bullet}\in \Ac(\EE)\cap \C(\AA)$, there is a map $\alpha^{\bullet}\colon A^{\bullet}\to E^{\bullet}$ in $\TT$ where $A^\bullet \in \Ac(\AA).$  In particular, $A^\bullet \cong E^\bullet$ in $\D(\AA).$
\end{corollary}

\begin{proof}
Choose $k\in \mathbb{Z}$. The complex $E^{\bullet}$ yields an $\AA$-resolution $E_L^{\bullet}\deflation \im(d_E^{k})$ of $\im(d_E^k)\in \EE$ where $E_L^{\bullet}$ is the brutal truncation of $E^{\bullet}$ in degree $k$. By \Cref{proposition:ResolutionLength}, we obtain a commutative diagram 
	\[\xymatrix{
		B^{\bullet} \ar@{->>}[r]\ar[d]^{f^{\bullet}} & \im(d_E^k)\ar@{=}[d]\\
		E_L^{\bullet} \ar@{->>}[r] & \im(d_E^k)\\
	}\] such that $\cone(f^{\bullet})\in \Acm(\AA)$ and such that $\im(d_B^{k-n})\in \AA$.  Consider now the following map
			\[\xymatrix@R=1pc@C=1pc{
			\Sigma^{-1 }\cone(f^{\bullet})\ar[d]^{g^{\bullet}}&& \dots\ar[r]& B^{l}\oplus E_L^{l-1}\ar[r]\ar[d]^{\begin{psmallmatrix}1&0\end{psmallmatrix}} & B^{l+1}\oplus E_L^{l}\ar[r]\ar[d]^{\begin{psmallmatrix}1&0\end{psmallmatrix}} & B^{l+2}\oplus E_L^{l+1}\ar[r]\ar[d]^{\begin{psmallmatrix}1&0\end{psmallmatrix}} & \dots  \\
			B^{\bullet}\ar[d]&& 							 \dots\ar[r]& B^l \ar[r]\ar[d]                & B^{l+1}\ar[r]\ar[d]             & B^{l+2}\ar[r]\ar[d]               & \dots  \\
			\cone(g^{\bullet})&&                     \dots\ar[r]& B^{l+1}\oplus E_L^{l}\oplus B^l\ar[r] & B^{l+2}\oplus E_L^{l+1}\oplus B^{l+1}\ar[r] & B^{l+3}\oplus E_L^{l+2}\oplus B^{l+2}\ar[r] &\dots \\
		}\]
		It is well known (and straightforward to verify) that $\cone(g) \cong \cone(1_{B^\bullet}) \oplus E_L^\bullet$ in $\C(\EE)$. 	By \Cref{Lemma:ConeOfAcyclic}, we have the following conflation
		\[\im(d_{\cone(g)}^{l})\inflation \im(d_{\cone(f)}^{l+1})\oplus B^{l+1}\deflation \im(d_B^{l+1})\]
for each $l \in \bZ.$  When $l = k-n-1$, then $\im(d_B^{l+1}) \in \AA$ and, by axiom \ref{PR1}, we have $\im(d_{\cone(g)}^{l}) \cong \im(d_E^l) \oplus B^{l+1} \in \AA.$  It follows that one can add a null-homotopic complex $C^{\bullet}\in \C(\AA)$ to $E^{\bullet}$ such that $C^{\bullet}\oplus E^{\bullet}\in \Ac(\AA)$. this concludes the proof.
\end{proof}

\begin{remark}
	Let $\AA$ be a preresolving subcategory in $\EE$ with $\resdim_\AA(\EE) = n \in \bN$. The above proof yields that $\im(d_E^l) \oplus B^{l+1}\in \AA$ for all $l\in \mathbb{Z}$, it follows that $\resdim_{\AA}(\im(d_E^{l}))\leq 1$.
\end{remark}

Using the above corollary, one readily verifies that the conditions of \cite[Proposition~7.2.1]{KashiwaraSchapira06} are satisfied. This completes the proof of \Cref{Theorem:UnboundedMainTheorem}.

%% file: Examples.tex
\section{Examples and applications}\label{Section:ExampleAndApplication}

\subsection{Examples}

\begin{example}
	Consider the Isbell category $\II$ (see \cite[Section~2]{Kelly69}), that is, $\II$ is the full subcategory of the category $\mathsf{Ab}$ of abelian groups containing no element of order $p^2$ (for a fixed prime $p$).	Clearly $\II\subseteq \mathsf{Ab}$ is deflation-closed (in fact, $\II\subseteq \mathsf{Ab}$ is closed under subobjects) and thus $\II$ is a deflation-exact category satisfying axiom \ref{R3+} (see \Cref{Proposition:DeflationClosedInheritsDeflationExactStructure}).
	
	Clearly, $\II$ contains all projectives and thus $\II\subseteq \mathsf{Ab}$ satisfies condition \ref{PR2}. Hence $\II\subseteq \mathsf{Ab}$ is preresolving. Furthermore, as $\mathsf{Ab}=\Mod(\mathbb{Z})$ global dimension one, we find that $\resdim_{\II}(\mathsf{Ab}) = 1.$  By \Cref{Theorem:PreResolvingDerivedEquivalence,Theorem:UnboundedMainTheorem}, we know that the natural functors $\Dstar(\II)\simeq \Dstar(\mathsf{Ab})$ are triangle equivalences (for $* \in \{-,b,\varnothing\}$).

\end{example}

\begin{remark}
	\begin{enumerate}
		\item Note that the Isbell category $\II$ is not an extension-closed subcategory of $\mathsf{Ab}$, as evidenced by the conflation $\mathbb{Z}/p\mathbb{Z}\inflation \mathbb{Z}/p^2\mathbb{Z}\deflation \mathbb{Z}/p\mathbb{Z}$. In particular, $\II\subseteq \mathsf{Ab}$ is not a (finitely) resolving subcategory.
		\item By \cite[Example~4.7]{BazzoniCrivei13}, the category $\II$ is pre-abelian but does not satisfy axioms \textbf{L1}, \textbf{L2} and \textbf{L3}.  In particular, $\II$ is not exact.
	\end{enumerate}
\end{remark}

\begin{example}  Let $\Gamma$ be an ordered group.  Given a $\Gamma$-filtered ring $FR$ and a subset $\Lambda\subseteq \Gamma$, one can define the category of $FR$-glider representations $\mathsf{Glid}_{\Lambda}(FR)$ (see \cite{CaenepeelVanOystaeyen19,HenrardvanRoosmalen20}).  The category of $FR$-glider representations is a subcategory of the category of $FR$-prefragments $\mathsf{Prefrag}_{\Lambda}(FR)$. By \cite[Corollary~5.11]{HenrardvanRoosmalen20}, $\mathsf{Glid}_{\Lambda}(FR)\subseteq \mathsf{Prefrag}_{\Lambda}(FR)$ is closed under subobjects (in particular, $\mathsf{Glid}_{\Lambda}(FR)\subseteq \mathsf{Prefrag}_{\Lambda}(FR)$ is deflation-closed), but in general not under extensions. Moreover, by \cite[Corollary~8.5]{HenrardvanRoosmalen20}, $\mathsf{Glid}_{\Lambda}(FR)$ has enough projectives and the projectives belong to $\mathsf{Prefrag}_{\Lambda}(FR)$. It follows that $$\resdim_{\mathsf{Glid}_{\Lambda}(FR)}(\mathsf{Prefrag}_{\Lambda}(FR)) = 1.$$

By \Cref{Theorem:PreResolvingDerivedEquivalence,Theorem:UnboundedMainTheorem}, we find that the natural functors $\Dstar(\mathsf{Glid}_{\Lambda}(FR))\simeq \Dstar(\mathsf{Prefrag}_{\Lambda}(FR))$ are triangle equivalences (for $* \in \{-,b,\varnothing\}$).  This recovers \cite[Proposition~8.6]{HenrardvanRoosmalen20}.
\end{example}

\subsection{Application: the exact hull}

Throughout this section, let $\EE$ denote a deflation-exact category satisfying axiom \ref{R3+}. Following \cite{HenrardvanRoosmalen19b} and \cite{Rosenberg11}, there is an exact category $\EE^{\mathsf{ex}}$, called the \emph{exact hull} of $\EE$, together with an exact embedding $j\colon\EE\to\EE^{\mathsf{ex}}$ which is $2$-universal among exact functors to exact categories. Explicitly, the exact hull $\EE^{\mathsf{ex}}$ can be constructed as the extension-closure of $i(\EE)$ in $\Db(\EE)$ and inherits a conflation structure from the triangulated structure of $\Db(\EE)$. The aim of this section is to provide an alternative proof for the fact that $j$ lifts to a derived equivalence on the bounded derived categories (see \cite[Theorem~7.15]{HenrardvanRoosmalen19b}) using \Cref{Theorem:PreResolvingDerivedEquivalence}. 

The following is \cite[Corollary~7.5]{HenrardvanRoosmalen19b}.

\begin{proposition}\label{Proposition:EveryObjectOfHullIsCokernelOfInflationInEE}
	For every $Z\in \EE^{\mathsf{ex}}$, there exists a conflation $X\inflation Y\deflation Z$ in $\EE^{\mathsf{ex}}$ such that $X,Y\in \EE$.
\end{proposition}

The above proposition yields that $\EE\subseteq \EE^{\mathsf{ex}}$ satisfies condition \ref{PR2}, in fact, $\resdim_\EE(\EE^{\mathsf{ex}}) \leq 1$. Thus, to show that $j\colon \EE\to \EE^{\mathsf{ex}}$ lifts to a derived equivalence $\D^{*}(\EE)\to \D^{*}(\EE^{\mathsf{ex}})$ (where $* \in \{-,b,\varnothing\}$), it suffices to show that $\EE\subseteq\EE^{\mathsf{ex}}$ is deflation-closed.

\begin{proposition}
A deflation-exact category $\EE$ satisying axiom \ref{R3+} lies deflation-closed in its exact hull $\EE^{\mathsf{ex}}.$
\end{proposition}

\begin{proof}
Let $X\stackrel{f}{\inflation} Y\stackrel{g}{\deflation} Z$ be a conflation in $\EE^{\mathsf{ex}}$ with $Y,Z\in \EE$. Viewing $g$ as a map between stalk complexes in $\EE^{\mathsf{ex}}\subseteq \Db(\EE)$, its cone is given by the two-term complex $Y\deflation Z$ with $Y$ in degree $-1$. Thus $X$ is isomorphic (in $\Db(\EE)$) to the two term complex $X^{\bullet}$ given by $Y\deflation Z$ with $Y$ in degree zero. On the other hand, \Cref{Proposition:EveryObjectOfHullIsCokernelOfInflationInEE} yields a conflation $A\inflation B\deflation X$ in $\EE^{\mathsf{ex}}$ with $A,B\in \EE$. It follows that $X$ is also isomorphic to the two-term complex $U^{\bullet}$ given by $A\inflation B$ with $A$ in degree $-1$.

Hence there is an isomorphism $U^{\bullet}\to X^{\bullet}$ in $\Db(\EE)$ which can be represented by a roof $U^{\bullet} \stackrel{\alpha}{\leftarrow}V^{\bullet}\stackrel{\beta}{\rightarrow} X^{\bullet}$ in $\Kb(\EE)$ where both $\alpha$ and $\beta$ are quasi-isomorphisms (here we used axiom \ref{R3+} and \Cref{Proposition:BasicPropertiesDerivedCategory} to see that all isomorphisms are quasi-isomorphisms). By \cite[Proposition~3.18]{HenrardvanRoosmalen19b}, we may assume that $U^{\bullet}$ is a two-term complex $U^{-1}\to U^0$. It follows that the cone of $\beta$ is the acyclic complex 
\[0\to U^{-1}\to U^0\to Y\xrightarrow{g} Z\to 0\]
in $\Kb(\EE)$. It follows that $g$ is a deflation in $\EE$ and thus $g$ admits a kernel $\ker(g)\in \EE$. One readily verifies that $\ker(g)\cong X$ in $\EE^{\mathsf{ex}}$ as $j$ is an exact embedding. 
\end{proof}

\begin{remark}
	We now have two conflation structures on $\EE$: the original conflation structure, and the conflation structure induced from $\ex{\EE}.$  It follows from \Cref{Proposition:BasicPropertiesDerivedCategory}.\eqref{item:R3Triangle}, and the assumption that $\EE$ satisfies axiom \ref{R3+}, that these conflation structures coincide.
\end{remark}

We have now proved the following theorem.

\begin{theorem}
	Let $\EE$ be a deflation-exact category. If $\EE$ satisfies axiom \ref{R3+}, then $\EE$ is a preresolving subcategory of the exact category $\EE^{\mathsf{ex}}$ such that $\resdim_{\EE}(\EE^{\mathsf{ex}})\leq 1$. In particular, the embedding $j\colon \EE\to \EE^{\mathsf{ex}}$ lifts to a triangle equivalence $\D^*(\EE)\to \D^*(\EE^{\mathsf{ex}})$, where $* \in \{-,b,\varnothing\}$.
\end{theorem}  

\begin{remark}
	In general, a deflation-exact category $\EE$ satisfying axiom \ref{R3+} need not be a resolving subcategory of $\EE^{\mathsf{ex}}$ (since then, as an extension-closed subcategory of an exact category, $\EE$ would be exact itself).  If $\EE$ does not satisfy axiom \ref{R3+}, $\EE$ need not be a deflation-closed subcategory of $\EE^{\mathsf{ex}}$ (as is illustrated in \cite[Example 7.6]{HenrardvanRoosmalen20Obs}).
\end{remark}

\section{A comparison to Keller's theorem}\label{Section:KellersCConditions}

For a preresolving subcategory $\AA$ of a deflation-exact category $\EE$, we have used axiom \ref{PR1} to show that the obvious induced conflation structure on $\AA$ is a deflation-exact structure.  In this section, we weaken the conditions required on $\AA$, following \cite[Theorem~12.1]{Keller96}. In this case, as $\AA$ is not deflation-closed in $\EE$, we cannot use \Cref{Proposition:DeflationClosedInheritsDeflationExactStructure}) to endow $\AA$ with a deflation-exact structure.  In \Cref{definition:AAConflations}, we present a different conflation structure on $\AA$ and we show in \Cref{Proposition:GeneralDeflationExactStructure} that this conflation structure endows $\AA$ with the structure of a deflation-exact category.  The derived category of $\AA$ in \Cref{Theorem:DeflationExactKellerTheorem} is taken with respect to this conflation structure. 

We recall the following (dualized) definition from \cite{Keller96}.

\begin{definition}
	Let $\EE$ be an exact category and let $\AA\subseteq \EE$ be a fully exact subcategory.
	\begin{enumerate}[align=left]
		\myitem{\textbf{C1}}\label{C1} For each $E\in \EE$, there exists a deflation $A\deflation E$ with $A\in \AA$.
		\myitem{\textbf{C2}}\label{C2} For each conflation $X\inflation Y \deflation C$ with $C\in \AA$, there exists a commutative diagram 
		\[\xymatrix{
			A\ar@{>->}[r]\ar[d] & B\ar@{->>}[r]\ar[d] & C\ar@{=}[d]\\
			X\ar@{>->}[r] & Y\ar@{->>}[r] & C
		}\] where the top row is a conflation in $\AA$.
	\end{enumerate}
\end{definition}

The following is \cite[Theorem~12.1]{Keller96}.

\begin{theorem}
	Let $\AA\subseteq \EE$ be a fully exact subcategory of an exact category. If $\AA\subseteq \EE$ satisfies axioms \ref{C1} and \ref{C2}, the induced functor $\Dm(\AA)\to \Dm(\EE)$ is a triangle equivalence.
\end{theorem}

\begin{remark}
	\begin{enumerate}
		\item Axiom \ref{C1} is equal to axiom \ref{PR2}.
		\item Keller remarks (see the paragraph above \cite[Theorem~12.1]{Keller96}) that if $\AA\subseteq \EE$ is a fully exact subcategory of an exact category $\EE$ and $\AA\subseteq \EE$ is deflation-closed and satisfies axiom \ref{C1}, then axiom \ref{C2} is implied.
		\item As Keller requires $\AA\subseteq \EE$ to be extension-closed, $\AA$ automatically inherits an exact structure from $\EE$ (in particular, $\D(\AA)$ is defined). 
	\end{enumerate}
\end{remark}

Inspired by the above remarks, we seek to weaken axiom \ref{PR1} in such a way that one obtains a natural deflation-exact structure on $\AA$ and such that we obtain an analogue of \Cref{Proposition:AcyclicComplexesOfDeflationClosedSubcategoryCoincide}.

The following definition and proposition yield a natural deflation-exact structure on $\AA\subseteq \EE$.

\begin{definition}\label{definition:AAConflations}
	Let $\EE$ be a deflation-exact subcategory and let $\AA\subseteq \EE$ be a full additive subcategory. A conflation $A\stackrel{f}{\inflation}B\stackrel{g}{\deflation}C$ in $\EE$ with $A,B,C\in \AA$ is called an \emph{$\AA$-conflation} if $g$ admits all pullbacks in $\AA$, i.e.~for any $h\colon D\to C$ in $\AA$, the pullback in $\EE$ of $h$ along $g$ belongs to $\AA$.
\end{definition}

\begin{proposition}\label{Proposition:GeneralDeflationExactStructure}
	Let $\EE$ be a deflation-exact category and let $\AA\subseteq \EE$ be a full additive subcategory. The $\AA$-conflations define a deflation-exact structure on $\AA$.
\end{proposition}

\begin{proof}
	Axiom \ref{R0} is clearly satisfied. Axiom \ref{R2} holds by definition of an $\AA$-conflation and the Pullback Lemma. It remains to show axiom \ref{R1}, to that end, consider two $\AA$-conflations $K\inflation B\deflation C$ and $L\inflation A\deflation B$. Consider the following commutative diagram in $\EE$:
	\[\xymatrix{
		L\ar@{=}[r]\ar@{>->}[d] & L\ar@{>->}[d] &\\
		P\ar@{>->}[r]\ar@{->>}[d] & A\ar@{->>}[r]\ar@{->>}[d] & C\ar@{=}[d]\\
		K\ar@{>->}[r] & B\ar@{->>}[r] & C\\
	}\] Here the lower-left square is a pullback square in $\EE$ (which exists by axiom \ref{R2}). As $A\deflation B$ is an $\AA$-deflation, the pullback $P$ belongs to $\AA$. The Pullback Lemma now shows that $P\inflation A\deflation C$ is an $\AA$-conflation.
\end{proof}

\begin{remark}
The conflation structure from \Cref{definition:AAConflations} is the largest conflation structure on $\AA$ such that the embedding $\AA \to \EE$ is conflation-exact.  Indeed, the class of $\AA$-conflations is the largest subclass for which axiom \ref{R2} is satisfied.
\end{remark}

We now modify axiom \ref{C2} accordingly.

\begin{definition}
	Let $\EE$ be a deflation-exact category and let $\AA\subseteq \EE$ be a full additive subcategory. We define the following axiom:
	\begin{enumerate}[align=left]
		\myitem{\textbf{C2}'}\label{C2'} For each conflation $X\inflation Y \deflation C$ with $C\in \AA$, there exists a commutative diagram 
		\[\xymatrix{
			A\ar@{>->}[r]\ar[d] & B\ar@{->>}[r]\ar[d] & C\ar@{=}[d]\\
			X\ar@{>->}[r] & Y\ar@{->>}[r] & C
		}\] where the top row is an $\AA$-conflation.
	\end{enumerate}
\end{definition}

The following proposition clarifies the difference between an $\AA$-conflation and an $\EE$-conflation with terms in $\AA$.

\begin{proposition}\label{Proposition:C2'YieldsAlmostDeflationClosed}
	Let $\AA$ be a full additive subcategory of a deflation-exact category $\EE$ such that axiom \ref{C2'} holds. Let $X\stackrel{i}{\inflation} Y\stackrel{p}{\deflation} Z$ be a conflation in $\EE$.
	\begin{enumerate}
		\item If $Y,Z\in \AA$, then there exists a $B\in \AA$ such that $X\oplus B\in \AA$.
		\item If $X,Y$ and $Z\in \AA$, then there exists an $B\in \AA$ such that 
		\[\xymatrix{X\oplus B\ar@{>->}[r]^-{\begin{psmallmatrix}i&0\\ 0&1_{B}\end{psmallmatrix}} & Y \oplus B\ar@{->>}[r]^-{\begin{psmallmatrix}p&0\end{psmallmatrix}} & Z }\] is an $\AA$-conflation.
	\end{enumerate}
\end{proposition}

\begin{proof}
Assume that $Z \in \AA.$  By axiom \ref{C2'}, we obtain an $\AA$-conflation $A\inflation B\deflation Z$ that maps to $(i,p)$ via maps $f\colon A\to X$ and $g\colon B\to Y$. By axiom \ref{R2}, we obtain the following pullback square:
		\[\xymatrix{
			& A\ar@{=}[r]\ar@{>->}[d] & A\ar@{>->}[d]\\
			X\ar@{>->}[r]\ar@{=}[d] & X\oplus B\ar@{->>}[r]^-{\begin{psmallmatrix}0&1_B\end{psmallmatrix}}\ar[d]_{\begin{psmallmatrix}i&g\end{psmallmatrix}}\ar@{->>}[d] & B\ar@{->>}[d]\ar@{-->}[ld]^g\\
			X\ar@{>->}[r] & Y\ar@{->>}[r] & Z
		}\] 
			\begin{enumerate}
				\item	If $Y\in \AA$, then using that $A \inflation B \deflation Z$ is an $\AA$-conflation, we find that $B\oplus X\in \AA$.  
				\item Assume that $X,Y\in \AA$. As $A\inflation B\deflation Z$ is an $\AA$-conflation and the $\AA$-conflations form a deflation-exact structure (\Cref{{Proposition:GeneralDeflationExactStructure}}), we can use \Cref{Proposition:LocalizationPaperProposition3.7} to obtain the $\AA$-conflation $$\xymatrix{X\oplus B\ar@{>->}[r]^-{\begin{psmallmatrix}i&g\\0&-1_B\end{psmallmatrix}} &  Y\oplus B\ar@{->>}[r]^-{\begin{psmallmatrix}p & pg\end{psmallmatrix}} & Z}.$$ The following diagram show that this $\AA$-conflation is isomorphic to the desired $\AA$-conflation:
				\[\xymatrix{
					X\oplus B\ar@{>->}[r]^-{\begin{psmallmatrix}i&g\\0&-1_B\end{psmallmatrix}}\ar[d]_{\begin{psmallmatrix}1_X&0\\0&-1_B\end{psmallmatrix}} &  Y\oplus B\ar@{->>}[r]^-{\begin{psmallmatrix}p & pg\end{psmallmatrix}}\ar[d]^{\begin{psmallmatrix}1_Y&g\\0&1_B\end{psmallmatrix}} & Z\ar@{=}[d]\\
					X\oplus B\ar@{>->}[r]_-{\begin{psmallmatrix}i&0\\0&1_B\end{psmallmatrix}} &  Y\oplus B\ar@{->>}[r]_-{\begin{psmallmatrix}p & 0\end{psmallmatrix}} & Z
				}\] This concludes the proof.\qedhere
			\end{enumerate}
\end{proof}

\begin{corollary}\label{Corollary:GoodReplacements}
	Let $\AA$ be a full additive subcategory of a deflation-exact category $\EE$ satisfying axiom \ref{C2'}.
	\begin{enumerate}
		\item If $E^{\bullet}\in \Acm(\EE)\cap \Cm(\AA)$, then $E^{\bullet}$ is homotopic to a complex $A^{\bullet}\in \Acm(\AA)$.
		\item If $E^{\bullet}\in \Acb(\EE)\cap \Cb(\AA)$, then $E^{\bullet}$ is homotopic to a complex $A^{\bullet}\in \Acb(\AA)$.
	\end{enumerate}
\end{corollary}

\begin{proof}
	 The is similar to \Cref{Proposition:AcyclicComplexesOfDeflationClosedSubcategoryCoincide} and uses \Cref{Proposition:C2'YieldsAlmostDeflationClosed} to add in a controlled manner objects to each $\im(d_E^{k})$ so that the resulting complex is the desired complex.
\end{proof}

	Let $\EE$ be a deflation-exact category and let $\AA\subseteq \EE$ be a full additive subcategory. The \emph{relative weak idempotent completion} $\widehat{\AA}_{\EE}$ is the full subcategory of $\EE$ generated by all kernels in $\EE$ of retractions $r\colon A\to B$ in $\EE$ with $A,B\in \AA$.
	
	If $\AA \subseteq \EE$ is an additive subcategory satisfying axiom \ref{C1}, then $\widehat{\AA}_{\EE} \subseteq \EE$ also satisfies axiom \ref{C1}.

\begin{corollary}\label{Corollary:RelativeWeakIdempotentCompletion}
	Let $\AA$ be a full additive subcategory of a deflation-exact category $\EE$ satisfying axiom \ref{C2'}.
		\begin{enumerate}
			\item The subcategory $\widehat{\AA}_{\EE}\subseteq \EE$ is deflation-closed and the embedding $\AA\to \widehat{\AA}_{\EE}$ lifts to a triangle equivalence $\Dm(\AA)\to \Dm(\widehat{\AA}_{\EE})$.
			\item Assume that $\AA \subseteq \EE$ satisfies axiom \ref{C1}.  If $\resdim_{\AA}(E)=n$ for $E\in \EE$, then $\resdim_{\widehat{\AA}_{\EE}}(E)\leq n$. 
		\end{enumerate}
\end{corollary}

\begin{proof}
	This follows from \Cref{Proposition:C2'YieldsAlmostDeflationClosed}.
\end{proof}

\begin{theorem}\label{Theorem:DeflationExactKellerTheorem}
	Let $\EE$ be a deflation-exact category and let $\AA\subseteq \EE$ be a full additive subcategory. If $\AA\subseteq \EE$ satisfies axioms \ref{C1} and \ref{C2'}, then the induced functor $\Dm(\AA)\to \Dm(\EE)$ is a triangle equivalence. Moreover, if $\resdim_{\AA}(E)<\infty$ for all $E\in \EE$, the functor $\Db(\AA)\to \Db(\EE)$ is essentially surjective.
\end{theorem}

\begin{proof}
	This follows directly from  \Cref{Lemma:ReplacementTechnique} and \Cref{Corollary:GoodReplacements}. Alternatively, the functor $\Dm(\AA)\to \Dm(\EE)$ factors as $\Dm(\AA)\to \Dm(\widehat{\AA}_{\EE})\to \Dm(\EE)$. The first map is a triangle equivalence by \Cref{Corollary:RelativeWeakIdempotentCompletion} and the second map is a triangle equivalence by \Cref{Theorem:PreResolvingDerivedEquivalence}.
\end{proof}